\documentclass[12pt, preprint]{amsart}

\usepackage[top=0.9in, bottom=1in, left=1in, right=1in]{geometry}
\usepackage{amsmath, amssymb, amsthm}
\usepackage{verbatim}
\usepackage{graphicx}
\usepackage{enumerate}
\usepackage{mathrsfs}
\usepackage{tcolorbox}
\usepackage{hyperref}
\usepackage{color}
\usepackage{tikz-cd}
\usepackage{csquotes}
\usepackage{hyperref}
\usepackage{bbm}
\usepackage{pdfpages}

\definecolor{darkgreen}{rgb}{0.4,0.0,0.0}

\newtheorem{thm}{Theorem}[section]
\newtheorem{prop}[thm]{Proposition}
\newtheorem{lem}[thm]{Lemma}
\newtheorem{cor}[thm]{Corollary}

\newcommand{\fA}{\mathfrak A}     \newcommand{\sA}{\mathcal A}
     \newcommand{\sB}{\mathcal B}

\newcommand{\fR}{\mathfrak R}     
     \newcommand{\sS}{\mathcal S}
     
     \newcommand{\sU}{\mathcal U}

\def\XXint#1#2#3{{\setbox0=\hbox{$#1{#2#3}{\int}$ }
\vcenter{\hbox{$#2#3$ }}\kern-.6\wd0}}

\newcommand{\C}{\mathbb{C}}

\newcommand{\N}{\mathbb{N}}

\theoremstyle{definition}
\newtheorem{definition}[thm]{Definition}

\newcommand\restr[2]{{
  \left.\kern-\nulldelimiterspace 
  #1 
  \vphantom{\big|} 
  \right|_{#2} 
  }}

\theoremstyle{remark}

\newtheorem{remark}[thm]{Remark}

\numberwithin{equation}{section}

\newtheoremstyle{ser}
{8pt}
{8pt}
{\it}
{}
{\sf}
{:}
{6mm}
{}

\newtheoremstyle{serr}
{8pt}
{8pt}
{\normalfont}
{}
{\sf}
{.}
{6mm}
{}

\theoremstyle{ser}
\newtheorem{claim}{Claim}

\theoremstyle{serr}
\newtheorem{claimpff}{Proof of Claim}

\theoremstyle{ser}
\newtheorem{qn}{Question}

\begin{document}

\title{On products of symmetries in von Neumann algebras}

\author{B V Rajarama Bhat}

\address{Statistics and Mathematics Unit\\
Indian Statistical Institute\\
 8th Mile, Mysore Road\\
  RVCE Post, Bengaluru\\
   Karnataka - 560 059, India}
\email{bhat@isibang.ac.in, bvrajaramabhat@gmail.com}

\author{Soumyashant Nayak}
\address{Statistics and Mathematics Unit\\
Indian Statistical Institute\\
 8th Mile, Mysore Road\\
  RVCE Post, Bengaluru\\
   Karnataka - 560 059, India}
\email{soumyashant@isibang.ac.in, soumyashant@gmail.com}

\author{P Shankar}
\address{ Department of Mathematics\\
Cochin University of Science and Technology\\
Kochi,
Kerala - 682022, India}
\email{shankarsupy@cusat.ac.in, shankarsupy@gmail.com}

\begin{abstract}
Let $\mathscr{R}$ be a type $II_1$ von Neumann algebra. We show that every unitary in $\mathscr{R}$ may be decomposed as the product of six symmetries (that is, self-adjoint unitaries) in $\mathscr{R}$, and every unitary in $\mathscr{R}$ with finite spectrum may be decomposed as the product of four symmetries in $\mathscr{R}$. Consequently, the set of products of four symmetries in $\mathscr{R}$ is norm-dense in the unitary group of $\mathscr{R}$. Furthermore, we show that the set of products of three symmetries in a von Neumann algebra $\mathscr{M}$ is {\bf not} norm-dense in the unitary group of $\mathscr{M}$. This strengthens a result of Halmos-Kakutani which asserts that the set of products of three symmetries in $\mathcal{B}(\mathscr{H})$, the ring of bounded operators on a Hilbert space $\mathscr{H}$, is not the full unitary group of $\mathcal{B}(\mathscr{H})$.
\end{abstract}






\maketitle

\section{Introduction}

In \cite{kadison-unitary}, Kadison initiated the study of infinite unitary groups (that is, the unitary group of factors not of type $I_n$ for $n \in \N$), acknowledging the difficulties arising from the fact they are not locally compact in the norm topology. Having shown various results of a topological flavor, he wonders about the algebraic nature of these groups with particular interest in the case of $II_1$ factors.  In response\footnote{This may be inferred from the comments in \cite[pg. 399]{kadison-unitary} and \cite{halmos-kakutani}.}, Halmos and Kakutani showed in \cite{halmos-kakutani} that in a type $I_{\infty}$ factor, every unitary may be decomposed into the product of four symmetries; We remind the reader that a {\it symmetry} or {\it reflection} is a synonym for a self-adjoint unitary. Moreover, they showed that for purely algebraic reasons the scalar unitary $\exp(\frac{2 \pi i}{3})I$, being a central element of order $3$, cannot be decomposed as the product of three symmetries.

It is natural to wonder whether similar conclusions hold for other von Neumann algebras. For $n \in \N$, since every symmetry in the type $I_n$ factor $M_n(\C)$ has determinant $\pm 1$, every product of symmetries in $M_n(\C)$ must be an $n \times n$ unitary matrix with determinant $\pm 1$. Conversely, in \cite[Theorem 3]{radjavi}, Radjavi showed that any $n  \times n$ unitary matrix with determinant $\pm 1$ can be decomposed as the product of four symmetries in $M_n(\C)$. In \cite{fillmore}, Fillmore showed that any unitary in a properly infinite von Neumann algebra can be decomposed as the product of four symmetries; in particular, this takes care of the case of type $I_{\infty}$, type $II_{\infty}$, and type $III$ von Neumann algebras. In their approach, Halmos and Kakutani, and Fillmore essentially use the well-known fact that every permutation of a non-empty set can be written as the product of two involutions (permutations of order two); In the context of permutations of an orthonormal basis $\sB$ of a Hilbert space $\mathscr{H}$, it tells us that every unitary on $\mathscr{H}$ that permutes some orthonormal basis of $\mathscr{H}$ (such as ``shift" operators) can be written as the product of two symmetries.  

From the work of Broise (see \cite{broise}), it is known that every unitary in a type $II_1$ factor is the product of finitely many symmetries. From Dowerk and Thom's refinement of Broise's result (see \cite[Theorem 3.1]{dowerk_thom}), we may infer that every unitary in a type $II_1$ factor may be decomposed as the product of $16$ symmetries. In this article, our main goal is to investigate the minimal such number in this context. Furthermore, we directly work with von Neumann algebras of type $II_1$ rather than factors of type $II_1$.

The key results of this article are summarized below.
\vskip 0.05in
\begin{enumerate}[(i)]
    \item {[}Theorem \ref{thm:mainthm1}{]} A unitary in a type $I_n$ von Neumann algebra $\mathscr{R}$ can be decomposed as the product of four symmetries in $\mathscr{R}$ if and only if its center-valued determinant (see Definition \ref{def:cv_det}) is a central symmetry.
    \item {[}Theorem \ref{thm:mainthm2}{]} In a type $II_1$ von Neumann algebra $\mathscr{R}$, every unitary is the product of six symmetries in $\mathscr{R}$.
    \item {[}Theorem \ref{thm:mainthm3}{]} In a type $II_1$ von Neumann algebra $\mathscr{R}$, every unitary with finite spectrum is the product of four symmetries in $\mathscr{R}$.
    \item {[}Theorem \ref{thm:mainthm4}{]} In a von Neumann algebra $\mathscr{M}$, the set of unitaries which can be decomposed as the product of three symmetries in $\mathscr{M}$ is {\bf not} norm-dense in the unitary group of $\mathscr{M}$.
\end{enumerate}
 
Before proceeding further, we set up some notation used throughout this article.
\vskip 0.1in
\noindent {\bf Notation}: The set of complex numbers of unit modulus is denoted by $S^1$. For $n \in \N$, we denote the set of $n \times n$ complex matrices by $M_n(\C)$ and the group of unitary matrices in $M_n(\C)$ by $U(n)$. The spectrum of an operator $T$ is denoted by $\mathrm{sp}(T)$. For a subset $A$ of a von Neumann algebra $\mathscr{M}$, we denote the norm-closure of $A$ by $A^{=}$. The symbol $\sim$ describes the Murray-von Neumann equivalence relation for projections in a given von Neumann algebra. Let $\mathscr{R}$ be a finite von Neumann algebra with center $\mathscr{C}$ and $\tau : \mathscr{R} \to \mathscr{C}$ denote the canonical center-valued trace. The restriction of $\tau$ to the set of projections in $\mathscr{R}$ is called the (center-valued) {\it dimension function} (see \cite{kadison-ringrose2}) which we denote by $\mathrm{dim}_{\tau}$.  For a von Neumann algebra $\mathscr{M}$, we denote the group of unitary operators in $\mathscr{M}$ by $\mathcal{U}(\mathscr{M})$, and the set of symmetries in $\mathcal{U}(\mathscr{M})$ by $\mathcal{S}(\mathscr{M})$.
\vskip 0.1in

 For a positive integer $n$, let $\mathcal{S}(\mathscr{M})^n$ denote the set of unitaries that can be decomposed as the product of $n$ symmetries in $\mathscr{M}$. Since the identity operator is a symmetry, it is straightforward to see that $$\mathcal{S}(\mathscr{M}) \subseteq \mathcal{S}(\mathscr{M})^2 \subseteq \cdots \subseteq \mathcal{S}(\mathscr{M})^n \subseteq \cdots \subseteq \bigcup_{n \in \N} \sS (\mathscr{M})^n \subseteq \mathcal{U}(\mathscr{M}).$$
Recasting our previous discussion in this notation, we observe that when $\mathscr{M}$ is a von Neumann algebra of type $I_{\infty}$, type $II_{\infty}$ or type $III$ we have $\mathcal{S}(\mathscr{M})^4 = \mathcal{U}(\mathscr{M})$. In Theorem \ref{thm:mainthm1} of this article, by appropriately adapting Radjavi's strategy (cf. \cite{radjavi}) for a type $I_n$ von Neumann algebra $\mathscr{M}$, we provide a simple characterization of $\sS (\mathscr{M})^4$ and note that $\sS (\mathscr{M})^4 = \bigcup_{n \in \N} \sS (\mathscr{M})^n$.

Let $\mathscr{R}$ be a type $II_1$ von Neumann algebra. In view of the previously discussed results for the other ``types" of von Neumann algebras, it is tempting to conjecture that $\sU (\mathscr{R}) = \sS (\mathscr{R})^4$. In Theorem \ref{thm:mainthm2}, we show that every unitary in $\mathscr{R}$ is the product of six symmetries in $\mathscr{R}$, that is, $\mathcal{U}(\mathscr{R}) = \sS (\mathscr{R})^6$. Furthermore, in Theorem \ref{thm:mainthm3}, we observe that every unitary in $\mathscr{R}$ with finite spectrum belongs to $\mathcal{S}(\mathscr{R})^4$. Since every unitary in $\mathscr{R}$ may be norm-approximated arbitrarily closely by a unitary in $\mathscr{R}$ with finite spectrum, we conclude that $\mathcal{S}(\mathscr{R})^4$ is norm-dense in $\mathcal{U}(\mathscr{R})$. The sequence $\{ \sS (\mathscr{R})^n \}_{n \in \N}$ of subsets of $\sU (\mathscr{R})$ transitions from a norm-dense subset of $\sU (\mathscr{R})$ at $n=4$ to all of $\sU (\mathscr{R})$ at $n=6$. This motivates us to investigate the evolution of the sets $\sS (\mathscr{R})^n$ starting from $n=1$.

Let $\mathscr{M}$ be a von Neumann algebra. In Proposition \ref{prop:prod-two-symm}, we characterize $\mathcal{S}(\mathscr{M})^2$ as the set of unitaries in $\mathscr{M}$ that are unitarily equivalent in $\mathscr{M}$ to their adjoint; in particular, the spectrum of a unitary in $\sS (\mathscr{M})^2$ is symmetric about the real axis. In Theorem \ref{thm:mainthm4}, we use the above characterization of $\sS (\mathscr{M})^2$ to show that a unitary with spectrum contained in exactly one of the four connected components of $S^1 \backslash \{ 1, i, -1, -i \}$ does not belong to $\mathcal{S}(\mathscr{M})^3$. Since such unitaries form a norm-open subset of $\mathcal{U}(\mathscr{M})$, we conclude that $\big( \mathcal{S}(\mathscr{M})^3 \big)^{=} \ne \sU (\mathscr{M})$, that is,  $\mathcal{S}(\mathscr{M})^3$ is {\bf not} norm-dense in $\mathcal{U}(\mathscr{M})$. This gives us a stronger version of the conclusion by Halmos and Kakutani in \cite{halmos-kakutani} that $\mathcal{S} \big( \mathcal{B}(\mathscr{H}) \big)^3 \ne \mathcal{U} \big( \mathcal{B}(\mathscr{H}) \big)$, where $\sB (\mathscr{H})$ is the ring of bounded operators acting on a Hilbert space $\mathscr{H}$.

At the very end, we discuss some open questions and conjectures for further exploration, the main question of interest being whether $\mathcal{S}(\mathscr{R})^4 = \mathcal{U}(\mathscr{R})$ for every type $II_1$ von Neumann algebra $\mathscr{R}$.

\section{Products of symmetries in a type $I_n$ von Neumann algebra}

For a positive integer $n$, let $\mathscr{R}$ be a type $I_n$ von Neumann algebra with center $\mathscr{C}$. It is well-known that $\mathscr{R} \cong M_n(\mathscr{C})$ (see \cite[Theorem 6.6.5]{kadison-ringrose2}). Since $\mathscr{C}$ is an abelian von Neumann algebra, there is a hyperstonean space $X$ such that $\mathscr{C} \cong C(X)$, the space of complex-valued continuous functions on $X$; hence, $\mathscr{R} \cong M_n \big( C(X) \big) \cong C \big( X; M_n(\C) \big)$, the space of $M_n(\C)$-valued continuous functions on $X$.

\begin{definition}
\label{def:cv_det}
Let $\det : M_n(\C) \to \C$ denote the usual determinant function. For $f \in C\big( X; M_n(\C) \big)$, the function $\det \circ f$ is in $C(X)$. We call the map $f (\in \mathscr{R}) \mapsto \det \circ f (\in \mathscr{C})$ as the {\it center-valued determinant} on $\mathscr{R}$ and denote it by $\det_c : \mathscr{R} \to \mathscr{C}$.
\end{definition}

\begin{remark}
\label{rmrk:detc}
It is straightforward to verify the following properties of $\det _c : \mathscr{R} \to \mathscr{C}$ from analogous properties of $\det : M_n(\C) \to \C$:
\begin{enumerate}[(i)]
    \item $\det_c(AB) = \det_c (A)  \det_c (B)$ for $A, B  \in \mathscr{R}$.
    \item $\det _c(A^*) = \det_c(A)^*$ for $A \in \mathscr{R}$.
\end{enumerate}
\end{remark}

\begin{remark}
\label{rmrk:detc_cont}
For $f = (f_{jk})_{1 \le j, k \le n} \in M_n \big( C(X) \big)$ with $f_{jk} \in C(X)$, note that
$$\mathrm{det}_c(f)(x) = \sum_{\sigma \in S_n} \Big( \mathrm{sgn}(\sigma) \prod_{j = 1}^n f_{j \sigma(j)}(x) \Big), \textrm{ for } x \in X.$$
From the above formula, it is straightforward to see that $\det _c : \mathscr{R} \to \mathscr{C}$ is norm-continuous.
\end{remark}

Using Radjavi's strategy in the case of $M_n(\C)$ (cf. \cite[Theorem 3]{radjavi} and Kadison's diagonalization theorem for normal matrices over von Neumann algebras, in the theorem below we completely characterize all unitaries in a type $I_n$ von Neumann algebra which can be written as the product of finitely many symmetries.

\begin{thm}
\label{thm:mainthm1}
\textsl{
Let $\mathscr{R}$ be a type $I_n$ von Neumann algebra with center $\mathscr{C}$. Then a unitary $U$ in $\mathscr{R}$ can be decomposed as the product of four symmetries if and only if $\det_c(U)$ is a symmetry in $\mathscr{C}$.
}
\end{thm}
\begin{proof}
As discussed above, let $X$ be the hyperstonean space such that $\mathscr{C} \cong C(X)$ and $\mathscr{R} \cong C\big( X; M_n(\C) \big)$. Note that every unitary in $\mathscr{R}$ corresponds to a $U(n)$-valued continuous function on $X$. Let $R : X \to U(n)$ be a symmetry in $\mathscr{R}$. Since $R^2 = I$ and $R = R^*$, from Remark \ref{rmrk:detc}, $\det_c (R)^2 = I$ and $\det_c(R)$ is self-adjoint. Thus $\det R(x) \in \{ \pm 1\}$ for all $x \in X$, or equivalently, the range of $\det_c (R) : X \to S^1$ is in $\{ \pm 1 \}$. Again from Remark \ref{rmrk:detc}, it is straightforward to see that if the unitary $U : X \to U(n)$ in $\mathscr{R}$ is the product of finitely many symmetries in $\mathscr{R}$, then the range of $\det_c(U) : X \to S^1$ is $\{ \pm 1 \}$ (or equivalently, $\det _c (U)$ is a symmetry in $\mathscr{C}$). This proves one direction of the assertion. 

Let $U$ be a unitary in $\mathscr{R} \cong M_n(\mathscr{C})$ such that $\det _c (U)$ is a symmetry in $\mathscr{C}$. By Kadison's diagonalization theorem (see \cite[Lemma 3.7, Theorem 3.19]{diag_kadison}), $U$ is unitarily equivalent to a diagonal matrix in $M_n(\mathscr{C})$. Hence we may assume that $U = \mathrm{diag} (\lambda _1, \ldots, \lambda _n)$ where $\lambda _j : X \to S^1$ is in $\mathscr{C} \cong C(X)$ for $1 \le j \le n$. Define unitary matrices $V$ and $W$ in $M_n(\mathscr{C})$ by
\begin{align*}
    V &:= \mathrm{diag} \Big( \lambda _1, \lambda_1 ^*, \ldots, \prod_{j=1}^{2 k + 1} \lambda_{j}, \Big( \prod_{j=1}^{2 k +1} \lambda_{j} \Big)^*, \ldots \Big),\\
    W &:= \mathrm{diag} \Big( 1, (\lambda _1 \lambda _2), (\lambda_1 \lambda_2) ^*, \ldots, \prod_{j=1}^{2 k} \lambda_{j}, \Big( \prod_{j=1}^{2 k} \lambda_{j} \Big)^*, \ldots \Big).
\end{align*}
Clearly $V, W \in \mathscr{R}$ and it is straightforward to verify by entry-wise multiplication of the diagonals that $U = VW$.
If $n$ is even (with $n=2\ell$), then all the entries in $V$ appear in conjugate pairs and the corresponding principal diagonal blocks may be decomposed as the product of two symmetries in $M_2(\mathscr{C})$ (see Remark \ref{rmrk:symprod_vNa}). The first diagonal entry of $W$ is the constant function $1$, which is a central symmetry, and the last diagonal entry is $\lambda _1 \lambda _2 \cdots \lambda _{2 \ell} = \det_c(U)$ which is a central symmetry by hypothesis. The remaining entries of $W$ appear in conjugate pairs and as observed before, the corresponding principal diagonal blocks may be decomposed as the product of two symmetries in $M_2(\mathscr{C})$. 

If $n$ is odd, then the last diagonal entry of $V$ is $\det_c(U)$, which is a central symmetry, and the remaining entries appear in conjugate pairs. For $W$, the first diagonal entry is $1$, which is a central symmetry, and the remaining entries appear in conjugate pairs. A similar argument as in the previous paragraph proves the desired result.

In summary, both $V$ and $W$ may be decomposed as the product of two symmetries in $M_n(\mathscr{C})$ and thus $U = VW$ may be decomposed as the product of four symmetries in $\mathscr{R}$.
\end{proof}

\begin{cor}
\textsl{
For a type $I_n$ von Neumann algebra $\mathscr{R}$, $\sS (\mathscr{R})^4$ is {\bf not} norm-dense in $\sU(\mathscr{R})$.
}
\end{cor}
\begin{proof}
Let $m$ be an integer coprime to $n$. Note that $\det_c \big( \exp( \pi i \frac{1}{m}) I \big) = \exp( \pi i \frac{n}{m}) I$, which is not a central symmetry. From Theorem \ref{thm:mainthm1}, $\exp( \pi i \frac{1}{m}) I \in \sU (\mathscr{R}) \backslash \sS (\mathscr{R})^4$.

 Since $\det _c : \mathscr{R} \to \mathscr{C}$ is a norm-continuous map (see Remark \ref{rmrk:detc_cont}) and the central symmetries form a norm-closed subset of $\mathscr{C}$, from Theorem \ref{thm:mainthm1} we conclude that $\sS (\mathscr{R})^4$ is norm-closed. Thus the assertion follows.
\end{proof}

\section{Products of symmetries in a type $II_1$ von Neumann algebra}

\subsection{Decomposition of a unitary as a product of six symmetries}
En route to Theorem \ref{thm:mainthm2}, we begin with some preparatory results.

\begin{lem}
\label{lem:inf_div1}
\textsl{
Let $\mathscr{M}$ be a von Neumann algebra acting on the Hilbert $\mathscr{H}$. Let $(E_n)_{n \in \N}$ be a sequence of mutually orthogonal projections in $\mathscr{M}$, and $(A_n)_{n \in \N}$ be a sequence of operators in $\mathscr{M}$ whose operator-norms are uniformly bounded above, that is, there exists an $\alpha > 0$ such that $\|A_n \| \le \alpha$ for all $n \in \N$. Then $\lim_{k \to \infty} \sum_{j=1}^{k} E_j A_j E_j (=: \sum_{j=1}^{\infty} E_j A_j E_j) $ exists in the strong-operator topology (and hence, in the weak-operator topology) and is an operator in $\mathscr{M}$.
}
\end{lem}
\begin{proof}
Define $A^{(k)} := \sum_{j=1}^{k} E_j A_j E_j$. Let $x \in \mathscr{H}$.
For positive integers $k_1 < k_2$, note that 
\begin{align*}
  \big\| \big( A^{(k_2)} - A^{(k_1)} \big) x \big\|^2 &= \Big\| \sum_{j=k_1+1}^{k_2} E_j A_j E_j x \Big\|^2 = \sum_{j=k_1 + 1}^{k_2} \|E_j A_j E_j x\|^2 \\
  &\le (\alpha ^2) \sum_{j=k_1 + 1}^{k_2} \|E_j x\|^2 
\end{align*}

Furthermore, we have $\sum_{j=1}^k \|E_j x\|^2$ is a Cauchy sequence converging to $\big\| \big( \sum_{j=1}^{\infty} E_j \big) x \big\|^2$. Thus $A^{(k)}x$ is Cauchy in norm and converges to a vector which we label $Ax$. In this notation, we observe that $A = \sum_{j=1}^{\infty} E_j A E_j$, is the strong-operator limit of the sequence of operators, $\big( A^{(k)} \big)_{k \in \N}$, in $\mathscr{M}$. As $\mathscr{M}$ is strong-operator closed, $A$ is an operator in $\mathscr{M}$.
\end{proof}

\begin{prop}
\label{prop:diag_prod}
\textsl{
Let $\mathscr{M}$ be a von Neumann algebra acting on the Hilbert $\mathscr{H}$ and $(E_n)_{n \in \N}$ be a sequence of mutually orthogonal projections in $\mathscr{M}$. Let $(A_n)_{n \in \N}, (B_n)_{n \in \N}$ be two sequences of operators in $\mathscr{M}$ whose operator-norms are uniformly bounded above, that is, there exists an $\alpha > 0$ such that $\|A_n \| \le \alpha, \|B_n\| \le \alpha$ for all $n \in \N$. Let $A := \sum_{j=1}^{\infty} E_j A_j E_j$, $B := \sum_{j=1}^{\infty} E_j B_j E_j$. Then $$AB = \sum_{j=1}^{\infty} E_j (A_j E_j B_j)E_j.$$}
\end{prop}
\begin{proof}
Clearly $\|A_k E_k B_k \| \le \alpha ^2$ for all $k \in \N$ so that by Lemma \ref{lem:inf_div1}, $\sum_{j=1}^{\infty} E_j (A_j E_j B_j)E_j$ may be interpreted as the strong-operator limit of the partial sums of the series.

Note that multiplication ($(S, T) \in \mathscr{M} \times \mathscr{M} \mapsto ST \in \mathscr{M}$) is separately continuous in the strong-operator topology. Thus for every $k \in \N$, we have $$E_k A_k E_k B = \sum_{j=1}^{\infty} (E_k A_k E_k) (E_j B_j E_j) = E_k A_k E_k B_k E_k,$$ and 
$$AB = \sum_{j=1}^{\infty} E_j A_j E_j B = \sum_{j=1}^{\infty} E_j (A_j E_j B_j) E_j.$$
\end{proof}

\begin{cor}
\label{cor:diag_uni_sa}
\textsl{
Let $\mathscr{M}$ be a von Neumann algebra acting on the Hilbert $\mathscr{H}$ and $( E_n )_{n \in \N}$ be a sequence of mutually orthogonal projections in $\mathscr{M}$ such that $\sum_{j=1}^{\infty} E_j = I$.
\begin{itemize}
    \item[(i)] Let $(U_n)_{n \in \N}$ be a sequence of operators in $\mathscr{M}$ such that $U_n E_n = E_n U_n = U_n$, and $U_n ^* U_n = U_n U_n ^* = E_n$; In other words, $U_n$ is a unitary operator in the von Neumann algebra $E_n \mathscr{M} E_n$ acting on $E_n(\mathscr{H})$. Then $\sum_{j=1}^{\infty} E_j U_j E_j$ is a unitary operator in $\mathscr{M}$;
    \item[(ii)] Let $(H_n)_{n \in \N}$ be a sequence of self-adjoint operators in $\mathscr{M}$ whose operator-norms are uniformly bounded above, that is, there exists an $\alpha > 0$ such that $\|H_n \| \le \alpha$ for all $n \in \N$. Then $\sum_{j=1}^{\infty} E_j H_j E_j$ is a self-adjoint operator in $\mathscr{M}$.
\end{itemize}
}
\end{cor}
\begin{proof} 
(i) From Lemma \ref{lem:inf_div1}, we observe that $U := \sum_{j=1}^{\infty} E_j U_j E_j, V := \sum_{j=1}^{\infty} E_j U_j ^* E_j$ are well-defined. From Proposition \ref{prop:diag_prod} (and the hypothesis), we conclude that $$UV = \sum_{j=1}^{\infty} E_j (U_j E_j U_j^*) E_j = \sum_{j=1}^{\infty} E_j (E_j U_j U_j^*) E_j = \sum_{j=1}^{\infty} E_j = I.$$
Similarly $VU = I$. Thus $U$ is a unitary operator in $\mathscr{M}$ and $V = U^*$.
\vskip 0.1in

\noindent (ii) This follows from the fact that the limit of self-adjoint operators, in the strong-operator topology, is self-adjoint.
\end{proof}

\begin{remark}
\label{rmrk:symprod_vNa}
Let $\mathscr{M}$ be a von Neumann algebra with identity $I$ and let $U_1$ be a unitary operator in $\mathscr{M}$. We observe that the unitary $$
\begin{bmatrix}
U_1 & 0\\
0 & U_1 ^*
\end{bmatrix} =
\begin{bmatrix}
0 & I\\
I & 0
\end{bmatrix}
\begin{bmatrix}
0 & U_1 ^*\\
U_1 & 0
\end{bmatrix}
$$ in $M_2(\mathscr{M})$ is a product of two symmetries in $M_2(\mathscr{M})$. 

\end{remark}

\begin{lem}
\label{lem:II1_masa_dec}
\textsl{
Let $\mathscr{R}$ be a type $II_1$ von Neumann algebra, and $U$ be a unitary in $\mathscr{R}$. Then for every $n \in \N$ there are projections $E_1, E_2, \ldots, E_n$ in $\mathscr{R}$ commuting with $U$ such that $E_1 \sim E_2 \sim \cdots \sim E_n$ and $E_1 + E_2 + \cdots + E_n = I$.}
\end{lem}
\begin{proof}
By \cite[Corollary 3.15]{diag_kadison}, for every positive integer $n$, each maximal abelian self-adjoint subalgebra of $\mathscr{R}$ contains $n$ orthogonal equivalent projections with sum $I$. Considering a maximal abelian self-adjoint subalgebra $\mathscr{A}$ of $\mathscr{R}$ containing $U$, the assertion follows immediately with projections $E_1, E_2, \ldots E_n \in \mathscr{A}$. 
\end{proof}

\begin{remark}
\label{rmrk:mat_alg_vNa}
Let $\mathscr{R}$ be a type $II_1$ von Neumann algebra acting on the Hilbert space $\mathscr{H}$. By \cite[Corollary 3.15]{diag_kadison}, for every $n \in \N$, there are equivalent mutually orthogonal projections $E_1, E_2, \ldots, E_n$ partitioning the identity, that is, $E_1 + E_2 + \cdots + E_n = I$. There is a system of matrix units $\{ V_{j,k} \}_{1 \le j, k \le n}$ where $V_{j,k}$ is a partial isometry in $\mathscr{M}$ with initial projection $E_k$ and final projection $E_j$ so that:
\begin{enumerate}
    \item $V_{j,j} = E_j$ for $1 \le j \le n$;
    \item $V_{j,k}^* = V_{k, j}$ for $1 \le j, k \le n$;
    \item $V_{j, k} V_{\ell, m} = \delta _{k, \ell} V_{j, m}$ for $1 \le j, k, \ell, m \le n$.
\end{enumerate}
Note that the mapping $\psi : \mathscr{R} \to M_n(E_1 \mathscr{R} E_1)$ given by $\psi(A) = (V_{1,j} A V_{k,1})_{1\le j, k \le n}$ is a $*$-isomorphism of von Neumann algebras. With this $*$-isomorphism in hand, we will take the liberty of viewing $\mathscr{R}$ as a matrix algebra over a type $II_1$ von Neumann algebra ($E_1 \mathscr{R} E_1$ acting on the Hilbert space $E_1(\mathscr{H})$, in this case) which helps avail various standard algebraic techniques involving matrices. Furthermore, as noted in Lemma \ref{lem:II1_masa_dec}, for a given unitary $U$ in $\mathscr{R}$, using Kadison's diagonalization theorem we may assume that $U$ is in diagonal form in this matrix representation.
\end{remark}

In the study of the decomposition of a unitary as the product of symmetries, the basic ingredient in previous investigations (see \cite{broise}, \cite{radjavi}, \cite{dowerk_thom}, and Theorem \ref{thm:mainthm1} above) is the identification of an appropriate ``basis" which ``diagonalizes" the unitary under consideration. This allows for the usage of appropriate matrix-algebraic techniques. Lemma \ref{lem:four_sym_eps} below sets the stage for such a (infinitary) ``diagonalization" process in the context of type $II_1$ von Neumann algebras to facilitate the use of the matrix decomposition mentioned in Remark \ref{rmrk:mat_alg_vNa}.

\begin{lem}
\label{lem:two_uni_prod_sym}
\textsl{
Let $\mathscr{R}$ be a type $II_1$ von Neumann algebra acting on the Hilbert space $\mathscr{H}$ and $\tau$ be the canonical center-valued trace on $\mathscr{R}$. Let $U, V$ be unitaries in $\mathscr{R}$. Then there is a projection $E \in \mathscr{R}$ with $\dim _{\tau} (E) = \frac{1}{2} I$, a unitary $V'$ in the von Neumann algebra $E\mathscr{R}E$ (acting on the Hilbert space $E(\mathscr{H})$), and symmetries $R_1, R_2 \in \mathscr{R}$ such that $$UV = R_1 R_2 U (V' E + I-E).$$
}
\end{lem}
\begin{proof}
Let $\mathscr{S}$ be a type $II_1$ von Neumann algebra such that $\mathscr{R} \cong M_2(\mathscr{S})$. By Kadison's diagonalization theorem, there are unitaries $V_1, V_2 \in \mathscr{S}$ and a unitary $W \in M_2(\mathscr{S})$ such that $V= W \mathrm{diag}(V_1, V_2) W^*$. Note that the projection $E := W \mathrm{diag}(0, I) W^* \in M_2(\mathscr{S})$ and the unitary $V' := W \mathrm{diag}(I, V_1 V_2) W^*$ in $M_2(\mathscr{S})$ commute. Furthermore, 
\begin{align*}
UV &= U \big( W \mathrm{diag}(V_1, V_2) W^* \big)\\
&= \big( U W \mathrm{diag}(V_1, V_1 ^*) W^* U^* \big) U \big( W \mathrm{diag}(I, V_1 V_2) W^* \big).\\
&= \big( U W \mathrm{diag}(V_1, V_1 ^*) W^* U^* \big) U (V'E + I-E).
\end{align*}
From Remark \ref{rmrk:symprod_vNa}, it is straightforward to see that $U W \mathrm{diag}(V_1, V_1 ^*) W^* U^*$ is the product of two symmetries in $M_2(\mathscr{S})$.  By appropriately interpreting these operators and our computations in the context of $\mathscr{R}$, we see that $\dim _{\tau} (E) = \frac{1}{2}I$ and the required conditions in the assertion are satisfied.
\end{proof}

\begin{lem}
\label{lem:four_sym_eps}
\textsl{
Let $\mathscr{R}$ be a type $II_1$ von Neumann algebra acting on the Hilbert space $\mathscr{H}$ and $\tau$ be the canonical center-valued trace on $\mathscr{R}$. Let $E_1 \in \mathscr{R}$ be a projection with $\dim _{\tau} (E_1) = \frac{2}{3} I$, and $B_1$ be a unitary operator in the von Neumann algebra $E_1 \mathscr{R} E_1$ (acting on the Hilbert space $E_1(\mathscr{H})$). Let $U$ be a unitary operator $\mathscr{R}$ commuting with $E_1$. Then there exist four symmetries $R_1, R_2, R_3, R_4 \in \mathscr{R}$, a projection $E_2 \in \mathscr{R}$ orthogonal to $E_1$ with $\dim _{\tau} (E_2) = \frac{1}{6} I$, and a unitary operator $B_2$ in the von Neumann algebra $E_2 \mathscr{R} E_2$ (acting on the Hilbert space $E_2(\mathscr{H})$) such that $$U = R_1 R_2 R_3 R_4 \big(B_1  E_1  + B_2 E_2 +  (I - E_1 - E_2) \big).$$
}
\end{lem}
\begin{proof}
By considering the unitary operator $U(B_1 ^*E_1 + I-E_1)$ in $\mathscr{R}$ (in lieu of $U$), without loss of generality, we may assume that $B_1 = E_1$.

Let $\mathscr{S}$ be a type $II_1$ von Neumann algebra such that $\mathscr{R} \cong M_3(\mathscr{S})$. We denote the identity of $\mathscr{S}$ by $I_{\mathscr{S}}$ and the center-valued trace on $\mathscr{S}$ by $\tau_{\mathscr{S}}$. Since $U$ and $E_1$ commute, without loss of generality, we may assume that $U$ and $E_1$ are both in diagonal form (via simultaneous diagonalization using Kadison's diagonalization theorem) so that $U = \mathrm{diag}(U_1, U_2, U_3)$, for unitaries $U_1, U_2, U_3$ in $\mathscr{S}$, and $E_1 = \mathrm{diag}(I_{\mathscr{S}}, I_{\mathscr{S}}, 0)$.

Using Lemma \ref{lem:two_uni_prod_sym} in the context of $\mathscr{S}$, we have a  projection $E \in \mathscr{S}$ with $\dim _{\tau_{\mathscr{S}}} (E) = \frac{1}{2} I_{\mathscr{S}}$, a unitary $V' \in \mathscr{S}$ commuting with $E$, and symmetries $S_1, S_2 \in \mathscr{S}$ such that 
$$U_3 = (U_2 ^* U_1 ^*) (U_1 U_2 U_3) = S_1 S_2 (U_2 ^* U_1 ^*) (V'E+I-E).$$
Thus we have
\begin{align*}
\mathrm{diag}(U_1, U_2, U_3) &= \mathrm{diag}(U_1, U_1 ^*, I_{\mathscr{S}}) \cdot  \mathrm{diag}(I_{\mathscr{S}}, U_1 U_2, U_2 ^* U_1^*) \cdot \mathrm{diag}(I_{\mathscr{S}}, I_{\mathscr{S}}, U_1 U_2 U_3)\\
&= \mathrm{diag}(U_1, U_1 ^*, S_1 S_2) \cdot  \mathrm{diag}(I_{\mathscr{S}}, U_1 U_2, U_2 ^* U_1^*) \cdot \mathrm{diag}(I_{\mathscr{S}}, I_{\mathscr{S}}, V'E+I-E).
\end{align*}
From Remark \ref{rmrk:symprod_vNa}, it is straightforward to see that the matrices, $$\mathrm{diag}(U_1, U_1 ^*, S_1 S_2),\;\; \mathrm{diag}(I_{\mathscr{S}}, U_1 U_2, U_2 ^* U_1^*),$$ can each be decomposed as the product of two symmetries in $M_3(\mathscr{S})$. We choose $E_2 := \mathrm{diag}(0, 0, E)$ and $B_2 := \mathrm{diag}(0, 0, V'E)$. By appropriately interpreting these operators and our computations in the context of $\mathscr{R}$, we see that $\dim _{\tau} (E_2) = \frac{1}{6}I$ and the required conditions in the assertion are satisfied.
\end{proof}

\begin{thm}
\label{thm:mainthm2}
\textsl{
Every unitary in a type $II_1$ von Neumann algebra $\mathscr{R}$ may be decomposed as the product of six symmetries in $\mathscr{R}$.
}
\end{thm}
\begin{proof}
Let $U$ be a unitary in $\mathscr{R}$, and $\mathscr{A}$ be a maximal abelian self-adjoint subalgebra of $\mathscr{M}$ containing $U$. Using \cite[Proposition 3.13, Corollary 3.14]{diag_kadison}, we may inductively choose a sequence of mutually orthogonal projections $F^{(n)}$ in $\mathscr{A}$ (and thus, commuting with $U$) such that $\dim _{\tau} \big( F^{(n)} \big) = \frac{3}{4^n} I.$ Note that $\sum_{n \in \N} F^{(n)} = I$.

For $n \in \N$, we choose $E_1 ^{(n)}$ to be a subprojection of $F^{(n)}$ with $\dim _{\tau} (E_1 ^{(n)}) = \frac{2}{3} \dim _{\tau} (F^{(n)}) = \frac{2}{4^n}I$. With the aim of bringing Lemma \ref{lem:four_sym_eps} into action, below we inductively define three relevant sequences (with some desirable properties):
\begin{enumerate}[(i)]
    \item a sequence of projections $\{ E_2 ^{(n)} \}_{n \in \N}$ such that $E_2 ^{(n)}$ is a subprojection of $F^{(n)}$ orthogonal to $E_1 ^{(n)}$ with $\dim _{\tau} (E_2 ^{(n)}) = \frac{1}{6} \dim _{\tau} (F^{(n)}) = \frac{2}{4^{n+1}} I$,
    \item a sequence $\{ B_1 ^{(n)}\}$ of elements in $E_1 ^{(n)} \mathscr{R} E_1^{(n)}$,
    \item a sequence $\{ B_2 ^{(n)} \}$ of elements in $E_2 ^{(n)} \mathscr{R} E_2 ^{(n)}$.
\end{enumerate}

Set $B_1 ^{(1)} := E_1 ^{(1)}.$ For $n \in \N$, based on our knowledge of $B_1 ^{(n)}$, we define $E_2 ^{(n)}, B_2 ^{(n)}$ as follows. For the unitary $UF^{(n)}$ in the type $II_1$ von Neumann algebra $F^{(n)} \mathscr{R} F^{(n)}$, by Lemma \ref{lem:four_sym_eps}, we have four symmetries $R_i^{(n)} \in F^{(n)} \mathscr{R} F^{(n)}$ for $1 \le i \le 4$, a projection $E_2 ^{(n)} \le F^{(n)}$ orthogonal to $E_1 ^{(n)}$ with $\dim _{\tau} \big( E_2 ^{(n)} \big) = \frac{1}{6} \dim _{\tau} \big( F^{(n)} \big)$ and a unitary operator $B_2^{(n)}$ in $E_2^{(n)} \mathscr{R} E_2^{(n)}$ such that 
$$UF^{(n)} = R_1^{(n)} R_2^{(n)} R_3^{(n)} R_4^{(n)} \big( B_1 ^{(n)} E_1^{(n)} + B_2^{(n)}E_2^{(n)} + F^{(n)} - E_1^{(n)} - E_2^{(n)}\big).$$

For $n \in \N$, based on our knowledge of $E_2 ^{(n)}$ and $B_2 ^{(n)}$, we define $B_1^{(n+1)}$ as follows. Since $\dim _{\tau} (E_2 ^{(n)}) = \frac{2}{4^{n+1}} I = \dim _{\tau} (E_1 ^{(n+1)})$, there is a partial isometry $V_n$ in $\mathscr{R}$ with initial projection $E_2^{(n)}$ and final projection $E_1 ^{(n+1)}$. Define $B_1^{(n+1)} := V_n (B_2 ^{(n)}) ^* V_n^*$ and note that it is a unitary in $E_1 ^{(n+1)} \mathscr{R} E_1 ^{(n+1)}$. 

For $n \in \N$, let $E_3^{(n)} := F^{(n)} - E_1 ^{(n)} - E_2^{(n)}$. In this notation, we have $$UF^{(n)} = R_1^{(n)} R_2^{(n)} R_3^{(n)} R_4^{(n)} \big( B_1 ^{(n)} E_1^{(n)} + B_2^{(n)}E_2^{(n)} + E_3 ^{(n)} \big),$$
where $F^{(n)} = E_1 ^{(n)} + E_2^{(n)} + E_3 ^{(n)}$.

Recall that $\{F^{(n)}\}$ is a sequence of mutually orthogonal projections partitioning the identity operator. From Corollary \ref{cor:diag_uni_sa}, we conclude that for $1 \le j \le 4$ the operator $R_j := \sum_{n \in \N} R_j ^{(n)}$ is a symmetry in $\mathscr{R}$, and $W := \sum_{n \in \N} \big( B_1 ^{(n)} E_1^{(n)} +  B_2^{(n)}E_2^{(n)} + E_3^{(n)} \big)$ is a unitary in $\mathscr{R}$.

\begin{claim}
\textsl{
The unitary $W$ is a product of two symmetries in $\mathscr{R}$.
}
\end{claim}
\begin{claimpff}
For $n \in \N$, we define $E^{(n)} := E_2 ^{(n)} + E_1 ^{(n+1)}$. Recall that $E_2 ^{(n)} \sim E_1 ^{(n+1)}$ with the equivalence implemented by the partial isometry $V_n$. Using the corresponding $2 \times 2$ system of matrix units, it is straightforward to see that $B_2 ^{(n)} E_2 ^{(n)} + B_1^{(n+1)} E_1 ^{(n+1)}$ is of the form $\mathrm{diag}(B, B^*)$, which by Remark \ref{rmrk:mat_alg_vNa}, is a product of two symmetries in $E^{(n)} \mathscr{R} E^{(n)}$. With $G^{(n)} := E_2 ^{(n)} + E_1 ^{(n+1)} + E_3 ^{(n+1)}$, we observe that $$W_n := B_2 ^{(n)} E_2 ^{(n)} + B_1^{(n+1)} E_1 ^{(n+1)} + E_3 ^{(n+1)},$$ is the product of two symmetries in $G^{(n)} \mathscr{R} G^{(n)}$. Furthermore, $\{ G^{(n)} \}$ is a sequence of mutually orthogonal projections such that 
\begin{align*}
\sum_{n \in \N} G^{(n)} &= \sum_{n \in \N} \big( E_2 ^{(n)} + E_1 ^{(n+1)} + E_3 ^{(n+1)} \big)\\
&= E_2 ^{(1)} + \sum_{n \ge 2} \big( E_1 ^{(n)} + E_2 ^{(n)} + E_3 ^{(n)} \big) \\
&= I - E_1^{(1)} - E_3 ^{(1)}.    
\end{align*}
Note that, 
\begin{align*}
    W &= \sum_{n \in \N} \big( B_1 ^{(n)} E_1^{(n)} +  B_2^{(n)}E_2^{(n)} + E_3 ^{(n)} \big)\\
    &= B_1 ^{(1)} E_1 ^{(1)} + E_3 ^{(1)} + \sum_{n \in \N} W_n\\
    &= E_1 ^{(1)} + E_3 ^{(1)}  + \sum_{n \in \N} W_n.
\end{align*}
From Corollary \ref{cor:diag_uni_sa}, we conclude that $W$ is the product of two symmetries in $\mathscr{R}$.
\end{claimpff}
Using Lemma \ref{lem:inf_div1}, we have $$R_1 R_2 R_3 R_4 W = \sum_{n \in \N} UF^{(n)} = U.$$ Thus $U$ can be decomposed as the product of six symmetries in $\mathscr{R}$.
\end{proof}

\subsection{Approximation of a unitary by products of four symmetries}

\begin{prop}
\label{prop:univ_Cstar_symprod}
\textsl{
Let $\theta \in [0, 1]$. Let $\mathfrak{A}$ be a unital $C^*$-algebra with identity $I$ with unitaries $U, V$ in $\fA$ such that $UV = \exp(2 \pi i \theta) VU$. Then there exist symmetries $R_1, R_2, R_3, R_4$ in $M_2(\mathfrak{A})$ such that $$\exp(\pi i \theta) \begin{bmatrix}
I & 0 \\
0 & I
\end{bmatrix} = R_1 R_2 R_3 R_4.$$
}
\end{prop}
\begin{proof}
Note that for the following four elements in $M_2(\fA)$,
$$R_1 := \begin{bmatrix}
0 & U\\
U^* & 0
\end{bmatrix},\; R_2 := \begin{bmatrix}
0 & I\\
I & 0
\end{bmatrix},\; R_3 := \exp(\pi i \theta) \begin{bmatrix}
0 & U^* V \\
UV^* & 0
\end{bmatrix},\; R_4 := \begin{bmatrix}
0 & V\\
V^* & 0
\end{bmatrix},$$ we have $$R_1 R_2 R_3 R_4 = \exp(\pi i \theta)
\begin{bmatrix}
I & 0\\
0 & I
\end{bmatrix}.$$ Clearly $R_1, R_2, R_4$ are symmetries and $R_3$ is a unitary. Since $UV = \exp(2 \pi i \theta) VU$, we have
 \begin{align*}
     R_3 ^2 &= \exp(2 \pi i \theta) \begin{bmatrix}
     U^*VUV^* & 0\\
     0 & UV^*U^*V
     \end{bmatrix}\\
     &= \exp(2 \pi i \theta) 
    \begin{bmatrix}
    \exp(-2 \pi i \theta) & 0\\
    0 & \exp(-2 \pi i \theta)
    \end{bmatrix}\\
    & = \begin{bmatrix}
    I & 0\\
    0 & I
    \end{bmatrix},
\end{align*}
and hence $R_3$ is also a symmetry.
\end{proof}

\begin{prop}
\label{prop:scalar_symprod}
\textsl{
Let $\mathscr{R}$ be a type $II_1$ von Neumann algebra. For every $\theta \in [0, 1]$, there are four symmetries $R_1, R_2, R_3, R_4$ in $\mathscr{R}$ such that $\exp(2 \pi i \theta) I = R_1 R_2 R_3 R_4$.
}
\end{prop} 
\begin{proof}
Let $\fR$ be the hyperfinite $II_1$ factor. Recall that $\fR$ unitally embeds into every type $II_1$ von Neumann algebra. Thus we need only prove the assertion for $\mathscr{R} =\fR$.
\vskip 0.1in
\noindent {\bf Case I:} $\theta$ is rational.
\vskip 0.1in
Let $\theta = \frac{m}{n}$.  There are Weyl unitaries $U, V$ in $M_n(\C)$ such that $UV = \exp(2 \pi i \theta) VU$. By Proposition \ref{prop:univ_Cstar_symprod}, $\lambda I_{2n}$ can be decomposed as the product of four symmetries in $M_{2n}(\C)$. Note that there is a unital embedding $M_{2n}(\C) \hookrightarrow \mathscr{R}$. Thus $\lambda I$ can be decomposed as the product of four symmetries in $\mathscr{R}$.
\vskip 0.1in
\noindent {\bf Case II:} $\theta$ is irrational.
\vskip 0.1in
Let $\sA _{\theta}$ denote the irrational rotation algebra corresponding to $\theta$. There are unitaries $U, V \in \sA _{\theta}$ such that $UV = \exp(2 \pi i \theta) VU$. By Proposition \ref{prop:univ_Cstar_symprod}, $\lambda I$ can be decomposed as the product of four symmetries in $M_2(\sA _{\theta})$. Note that there is a unital embedding $M_{2}(\sA_{\theta}) \hookrightarrow \mathscr{R}$. Thus $\lambda I$ can be decomposed as the product of four symmetries in $\mathscr{R}$.
\end{proof}

\begin{thm}
\label{thm:mainthm3}
{\sl Let $\mathscr{R}$ be a type $II_1$ von Neumann algebra and $U$ be a unitary operator in $\mathscr{R}$ with finite spectrum. Then there are four symmetries $R_1, R_2, R_3, R_4$ in $\mathscr{R}$ such that $U = R_1 R_2 R_3 R_4$.
}
\end{thm}
\begin{proof}
Let the spectrum of $U$ be $\{ \lambda_1, \cdots, \lambda_n \}.$ There are mutually orthogonal projections $E ^{(1)}, \cdots, E ^{(n)}$ such that $U = \sum_{k=1}^n \lambda_k E ^{(k)}$ and $\sum_{k=1}^n E ^{(k)} = I$. For $1 \le k \le n$, the von Neumann algebra $E ^{(k)} \mathscr{R} E ^{(k)}$ acting on $E ^{(k)}(\mathscr{H})$ is a type $II_1$ von Neumann algebra. From Proposition \ref{prop:scalar_symprod}, there are four symmetries $R_{1}^{(k)}, R_{2}^{(k)}, R_{3}^{(k)}, R_{4}^{(k)}$ in $E^{(k)} \mathscr{R} E^{(k)}$ such that $$\lambda_k E^{(k)} = R_{1}^{(k)} R_{2}^{(k)} R_{3}^{(k)} R_{4}^{(k)}, 1 \le k \le n.$$  We identify $\oplus_{k=1}^n E^{(k)} \mathscr{R} E^{(k)}$ in the canonical manner as a von Neumann subalgebra of $\mathscr{R}$. For $1 \le j \le 4$, we define $R_j := \sum_{k=1}^n R_{j}^{(k)}$. Note that $R_j$ is a symmetry in $\mathscr{R}$ as it is self-adjoint and $R_j^2 = \sum_{k=1}^n (R_{j}^{(k)})^2 = \sum_{k=1}^n E^{(k)} = I$, and $U = R_1 R_2 R_3 R_4.$
\end{proof}

\begin{cor}
\textsl{
Let $\mathscr{R}$ be a type $II_1$ von Neumann algebra. Then every unitary in $\mathscr{R}$ can be approximated in norm by a sequence of unitaries in $\mathcal{S}(\mathscr{R})^4$, that is, $\mathcal{U}(\mathscr{R}) = \big( \mathcal{S}(\mathscr{R})^4 \big)^{=}$.
}
\end{cor}
\begin{proof}
Let $W$ be any unitary operator in $\mathscr{R}$. By \cite[Theorem 5.2.5]{kadison-ringrose1}, there is a self-adjoint operator $H$ in $\mathscr{R}$ with spectrum in $[0, 1]$ such that $\exp(2 \pi i H) = W$. Using the spectral theorem, $H$ may be approximated in norm by a sequence of self-adjoint operators $\{ H_n \}_{n \in \mathbb{N} }$ in $\mathscr{R}$ which have finite spectrum contained in $[0, 1]$. Let $H_n := \sum_{j=1}^n \lambda _j E_j$. From Theorem \ref{thm:mainthm3}, for $n \in \mathbb{N}$, $W_n := \exp(2 \pi i H_n)$ is the product of four symmetries. As the function $t \rightarrow \exp(2 \pi i t)$ is continuous from $[0, 1]$ to $S^1$, the sequence of unitaries $\{ W_n \}_{n \in \mathbb{N}}$ converges to $W$ in norm.  
\end{proof}

\section{Products of Symmetries in a von Neumann algebra}

In this section, $\mathscr{M}$ denotes a von Neumann algebra acting on the Hilbert space $\mathscr{H}$. We investigate the evolution of the sequence $\{ \sS (\mathscr{M})^n \}_{n \in \N}$.

\begin{prop}[cf. {\cite[Theorem 3]{radjavi_williams}}]
\label{prop:prod-two-symm}
{\sl Let $\mathscr{M}$ be a von Neumann algebra. A unitary $U$ in $\mathscr{M}$ may be decomposed as the product of two symmetries in $\mathscr{M}$ if and only if there is a unitary $W$ in $\mathscr{M}$ such that $U^* = W^*UW$, that is, $U$ and $U^*$ are unitarily equivalent in $\mathscr{M}$. }
\end{prop}
\begin{proof}
Let $U = ST$ where $S, T$ are symmetries in $\mathscr{M}$. Since $U^* = TS = S(ST)S = SUS$, we observe that $U$ and $U^*$ are unitarily equivalent with the unitary equivalence implemented by $S$.

Conversely, let $W$ be a unitary in $\mathscr{M}$ such that $U^* = W^*UW$. By taking adjoint on both sides, we have $U = W^* U^* W$ which implies that $U^* = WUW^*$. Thus $$W^2U^* = W^2 (W^*UW) = WUW = (WUW^*) W^2 =  U^* W^2.$$ Thus $W^2$ commutes with $U^*$ (and hence also with its multiplicative inverse $U$). Let $\mathscr{N}$ be the von Neumann subalgebra of $\mathscr{M}$ generated by $\{I, W^2 \}$. Clearly $\{ U, U^*, W, W^* \} \subseteq \mathscr{N}' \cap \mathscr{M}$, the relative commutant of $\mathscr{N}$ in $\mathscr{M}$. By \cite[Theorem 5.2.5]{kadison-ringrose1}, there is a positive operator $H$ in $\mathscr{N}$ such that $W^2 = \exp(2 \pi i H)$. Note that $V := \exp( \pi i H)$ is a unitary operator in $\mathscr{N}$ such that $V^2 = W^2$. Since $V \in \mathscr{N}$, we have that $V$ and $V^*$ commute with each of $U, U^*, W, W^*$. Thus $S = V^* W$ is a symmetry and $US = SU^* $. It follows that $U = S(US)$ is the product of two symmetries in $\mathscr{M}$.
\end{proof}

\begin{remark}
\label{rmrk:prod_two_symm}
In particular, Proposition \ref{prop:prod-two-symm} implies that the spectrum of an element in $\mathcal{S}(\mathscr{M})^2$ is symmetric about the real axis.
\end{remark}

\newpage
\begin{lem}
\label{lem:2by2mat}
\textsl{
\begin{enumerate}[(i)]
    \item Let $\mathscr{M}$ be a von Neumann algebra with {\bf no} direct summand of type $I_{2k-1}$ ($k \in \N$) in its type decomposition. Then there are two mutually orthogonal equivalent projections $E_1, E_2$ in $\mathscr{M}$ such that $E_1 + E_2 = I$.
    \item A von Neumann algebra $\mathscr{M}$ has {\bf no} direct summand of type $I_{2k-1}$ ($k \in \N$) in its type decomposition if and only if $\mathscr{M} \cong M_2(\mathscr{S})$ for some von Neumann algebra $\mathscr{S}$.
\end{enumerate}
}
\end{lem}
\begin{proof}
(i) Since $\mathscr{M}$ is a direct sum of a properly infinite von Neumann algebra and von Neumann algebras of types $I_{2k}$ and $II_1$, it suffices to consider the three cases in which either $\mathscr{M}$ is properly infinite, or type $I_{2k}$ ($k \in \N$) or type $II_1$. For a properly infinite von Neumann algebra, the assertion follows from \cite[Lemma 6.3.3]{kadison-ringrose2}. For a type $II_1$ von Neumann algebra or a type $I_{2k}$ von Neumann algebra ($k \in \N$), the assertion follows from \cite[Theorems 8.4.3, 8.4.4]{kadison-ringrose2}.
\vskip 0.1in

\noindent (ii) If $\mathscr{M}$ is of the form $M_2(\mathscr{S})$ for a von Neumann algebra $\mathscr{S}$, then clearly $\mathscr{M}$ has no direct summand of type $I_{2k-1}$ ($k \in \N$) in its type decomposition. The converse follows from part (i) and Remark \ref{rmrk:mat_alg_vNa}.
\end{proof}

\begin{cor}
\label{cor:prod_three_sym}
Let $\mathscr{M}$ be a von Neumann algebra with {\bf no} direct summand of type $I_{2k-1}$ ($k \in \N$) in its type decomposition, and $\alpha \in S^1 \subset \mathbb{C}$. The scalar unitary $\alpha I$ in $\mathscr{M}$ can be decomposed as the product of three symmetries in $\mathscr{M}$ if and only if $\alpha \in \{1, i, -1, -i \}$. 
\end{cor}
\begin{proof}
Let $R_1, R_2, R_3$ be symmetries in $\mathscr{M}$ such that $\xi I = R_1 R_2 R_3$. We have $\xi R_3 = R_1 R_2 \in \mathcal{S}(\mathscr{R})^2$. If $\xi \ne \pm 1$, we must have $\overline{\xi} = - \xi$ in case of which we have $\xi = \pm i $.

Since $I, -I$ are symmetries, we need only prove the assertion for $\alpha = \pm i $. By Lemma \ref{lem:2by2mat}, (ii), we may assume that $\mathscr{M}\cong M_2(\mathscr{S})$ for some von Neumann algebra $\mathscr{S}$ with identity $I_{\mathscr{S}}$. Note that
$$
i I = 
\begin{bmatrix}
i I_{\mathscr{S}} & 0\\
0 & i I_{\mathscr{S}}
\end{bmatrix} = 
\begin{bmatrix}
0 & I_{\mathscr{S}}\\
I_{\mathscr{S}} & 0
\end{bmatrix}
\begin{bmatrix}
0 & -i I_{\mathscr{S}}\\
i I_{\mathscr{S}} & 0
\end{bmatrix}
\begin{bmatrix}
I_{\mathscr{S}} & 0\\
0 & -I_{\mathscr{S}}
\end{bmatrix}
$$
is a product of three symmetries in $M_2(\mathscr{S}) \cong \mathscr{M}$. We may deduce a similar decomposition of $-i I$.
\end{proof}

Note that for an odd positive integer $n$, the scalar unitary matrix $iI \in M_n(\C)$ has imaginary determinant and thus cannot be decomposed into the product of any finite number of symmetries, let alone the product of three symmetries.

For $1 \le k \le 4$, let $C_k := \big\{ \exp (2 \pi i \alpha) : \frac{k-1}{4} < \alpha < \frac{k}{4} \big\}$ denote the four connected components of $S^1 \backslash \{ 1, i, -1, -i \}$.

\begin{prop}
\label{prop:prod_three_sym}
\textsl{
Let $U$ be a unitary operator in $\mathcal{B}(\mathscr{H})$ such that $\mathrm{sp}(U) \subset C_k$ for some $k \in \{1, 2, 3, 4 \}$. Then $U$ cannot be decomposed as the product of three symmetries.
}
\end{prop}
\begin{proof}
Let $\xi := \exp \big( 2 \pi i (\frac{2k-1}{8}) \big)$. Since $U$ is a normal operator with $\mathrm{sp}(U) \subset C_k$, from the spectral theorem we have that $$\|U - \xi I \| < \max_{\ell = k, k-1} \Big\{ \Big| \exp \Big(2 \pi i \big(\frac{\ell}{4}\big) \Big) - \exp \Big(2 \pi i \big(\frac{2k-1}{8}\big) \Big) \Big| \Big\} = 2 \sin \frac{\pi}{8}.$$ Let, if possible, $U = R_1 R_2 R_3$ for symmetries $R_1, R_2, R_3$. We note that $\| R_1 R_2 - \xi R_3 \| = \|(U - \xi I)R_3 \| < 2 \sin \frac{\pi}{8}$. As $R_3$ is a symmetry, $\mathrm{sp}(\xi R_3) \subseteq \{ \xi, -\xi \}$. 

Recall that if $A$ and $B$ are normal operators and $\lambda \in \mathrm{sp}(A)$, then the distance from $\lambda$ to $\mathrm{sp}(B)$ is at most $\| A - B \|$. Thus each element of $\mathrm{sp}(R_1 R_2)$ is atmost a distance of $2 \sin \frac{\pi}{8}$ away from $\{ \xi, -\xi \}$. We conclude that either $\mathrm{sp}(R_1 R_2)$ is contained inside $C_1 \cup C_3$ or inside $C_2 \cup C_4$ depending on $k$.

\begin{center}
\includegraphics[scale=0.32]{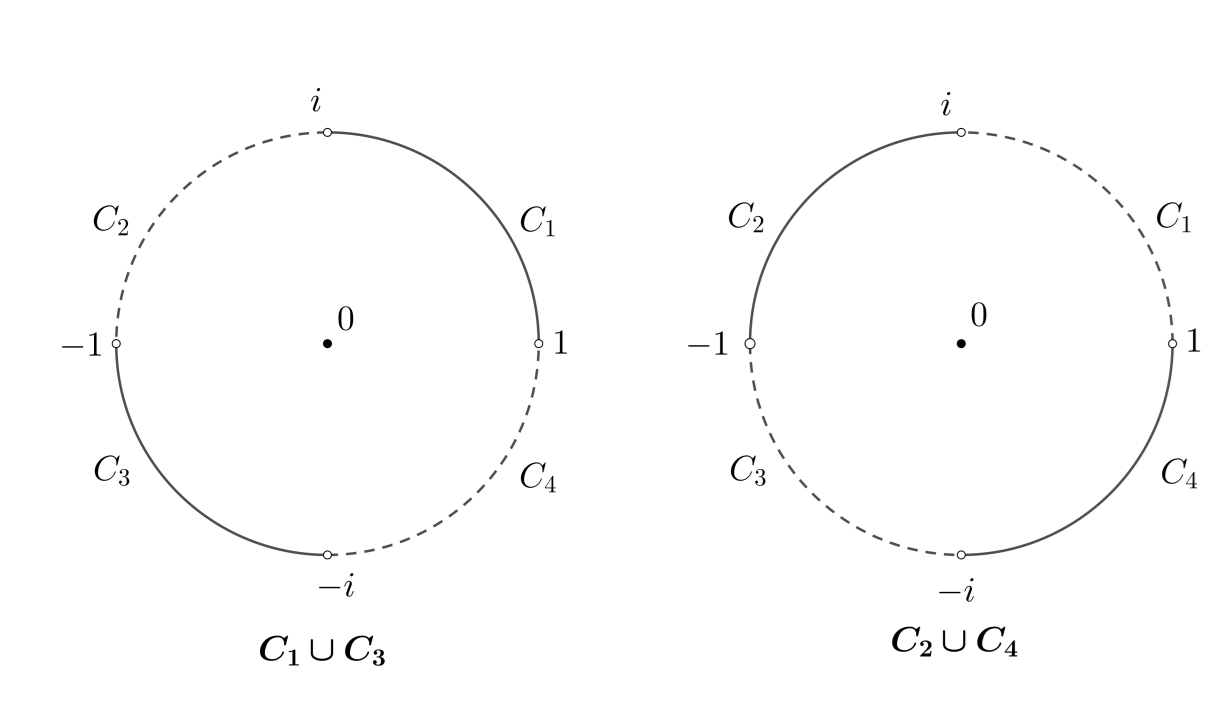}
\end{center}

For $z \in S^1$, we observe that $z \in C_1 \cup C_3$ if and only if $\overline{z} \notin C_1 \cup C_3$, and a similar conclusion also holds for $C_2 \cup C_4$. In view of Remark \ref{rmrk:prod_two_symm}, this leads to a contradiction. Thus our original assumption that $U$ is the product of three symmetries must be incorrect.
\end{proof}

\begin{remark}
Le $\mathscr{M}$ be a von Neumann algebra with identity operator $I$. Since the spectrum of $\exp( \frac{2 \pi i}{3}) I$ is $\big\{ \exp( \frac{2 \pi i}{3} ) \big\}$ which is contained in $S^1 \backslash \{ 1, i, -1 -i\}$, from Proposition \ref{prop:prod_three_sym} we conclude that $\exp( \frac{2 \pi i}{3}) I$ cannot be decomposed as the product of three symmetries in $\mathscr{M}$. In \cite[pg. 78]{halmos-kakutani}, Halmos and Kakutani show this using purely algebraic means whereas we view it through an operator-theoretic lens.
\end{remark}

\begin{thm}
\label{thm:mainthm4}
\textsl{
Let $\mathscr{M}$ be a von Neumann algebra. The set of unitaries with spectrum contained inside $C_k$ for some $k \in \{ 1, 2, 3, 4 \}$ is an open subset of $\mathcal{U}(\mathscr{M})$ in the uniform topology. Thus $\big( \sS (\mathscr{M})^3 \big)^{=} \ne \sU (\mathscr{M})$.
}
\end{thm}
\begin{proof}
Without loss of generality by multiplication, we may consider a unitary $U$ in $\mathscr{M}$ such that $\mathrm{sp}(U) \subset C_1$. As $\mathrm{sp}(U)$ is compact, we have $\varepsilon := \mathrm{dist} \big( \mathrm{sp}(U), \{ 1, i \} \big) > 0$. If $V$ is a unitary in $\mathscr{M}$ such that $\| V - U \| < \varepsilon$, then $V$ has spectrum contained in $C_1$. This proves that the set of unitaries with spectrum contained inside $C_1$ is an open set. Further $C_1$ contains $\exp ( i \frac{\pi}{4} )I$ and is thus a non-empty open subset of $\mathcal{U}(\mathscr{M}) \backslash \mathcal{S}(\mathscr{M})^3 $ by Proposition \ref{prop:prod_three_sym}. As a result, we have $\big( \sS (\mathscr{M})^3 \big)^{=} \ne \mathcal{U}(\mathscr{M})$.
\end{proof}

\begin{thm}
\textsl{
Let $\mathscr{M}$ be a von Neumann algebra with no direct summand of type $I_n$ (for $n \in \N$) in its type decomposition. Then $\sS (\mathscr{M})^6 = \sU (\mathscr{M})$, and $\sS (\mathscr{M})^4$ is norm-dense in $\sU (\mathscr{M})$.
}
\end{thm}
\begin{proof}
The assertion follows from Theorem \ref{thm:mainthm2}, Theorem \ref{thm:mainthm3} and \cite[pg. 900]{fillmore}.
\end{proof}

Let $\mathscr{M}$ be a von Neumann algebra with no direct summand of type $I_n$ (for $n \in \N$) in its type decomposition. From the results of this section, we have the following strict inclusions $\sS (\mathscr{M})^{=} \subset \big( \sS (\mathscr{M})^2 \big)^{=} \subset \big( \sS (\mathscr{M})^3 \big)^{=} \subset \big( \sS (\mathscr{M})^4 \big)^{=} = \sU (\mathscr{M}).$ Since $\big( \sS (\mathscr{M})^4 \big)^{=} = \big( \sS (\mathscr{M})^n \big)^{=}$ for all $n \ge 4$ and $\sS (\mathscr{M})^6 = \big( \sS (\mathscr{M})^4 \big)^{=}$, it looks plausible to conjecture that $\sS (\mathscr{M})^4 = \sS (\mathscr{M})^6 = \sU (\mathscr{M}). $

We end our discussion with a few open questions as fodder for further investigation. For a type $II_1$ von Neumann algebra $\mathscr{R}$, let $\mathfrak{N}_{\textrm{sym}}(\mathscr{R})$ denote the smallest positive integer $n$ such that $\mathcal{S}(\mathscr{R})^n = \mathcal{U}(\mathscr{R})$. By Theorem \ref{thm:mainthm2} and $\ref{thm:mainthm4}$, $4 \le \mathfrak{N}_{\textrm{sym}}(\mathscr{R}) \le 6$.

\begin{qn}
\textsl{
Let $\fR$ be the hyperfinite $II_1$ factor. Is 
$\mathfrak{N}_{\textrm{sym}}(\fR) = 4$?
}
\end{qn}

There are two extreme ends to the localization properties of the spectrum of unitary operators in a $II_1$ factor $\mathscr{R}$. Scalar unitaries correspond to the case where the spectrum is highly localized (supported at a point of $S^1$). Haar unitaries correspond to the case where the spectrum is highly non-localized (uniformly spread over $S^1$). We have shown that every scalar unitary in $\mathscr{M}$ is the product of four symmetries in $\mathscr{R}$. This brings us to our next question.
\begin{qn}
\textsl{
Can every Haar unitary in a $II_1$ factor $\mathscr{R}$ be decomposed as the product of four symmetries in $\mathscr{R}$?
}
\end{qn}

\begin{qn}
\textsl{
Let $\mathscr{R}$ be a $II_1$ von Neumann algebra. Is $\mathfrak{N}_{\textrm{sym}}(\mathscr{R}) = 4$? If $\mathscr{R} \cong M_2(\mathscr{R})$, is $\mathfrak{N}_{\textrm{sym}}(\mathscr{R}) = 4$?
}
\end{qn}

\noindent {\bf Acknowledgments:} B V Rajarama Bhat is supported by J C Bose Fellowship of SERB (Science and Engineering Research Board, India). Soumyashant Nayak is supported by Startup Research Grant (SRG/2021/002383) of SERB (Science and Engineering Research Board, India), and Startup-Grant by the Indian Statistical Institute (dated June 17, 2021). Shankar P is supported by Seed Money for New Research Initiatives (order No.CUSAT/ PL(UGC).A1/1112/2021) of CUSAT (Cochin University of Science and Technology, India).

\bibliographystyle{plain}
\bibliography{reference}

\end{document}